\documentclass[12pt]{amsart}
\usepackage{amsmath, amssymb, amsthm, enumerate, setspace, enumerate, epigraph}
\usepackage[T1]{fontenc}
\usepackage[vcentermath]{youngtab} 
\newtheorem{proposition}{Proposition}
\newtheorem{corollary}{Corollary}
\theoremstyle{remark}

\newtheorem*{remark}{Remark}

\providecommand{\abs}[1]{\lvert#1\rvert}
\DeclareRobustCommand{\stirling}{\genfrac\{\}{0pt}{}}
\begin{document}

\hspace{0.2in}
\title{Two Propositions Involving the Standard Representation of $S_n$}
\author{Shanshan Ding}
\address{Department of Mathematics, University of Pennsylvania, 209 South 33rd Street, Philadelphia, PA 19104} \email{shanshand@math.upenn.edu}
\begin{abstract}We present here two standalone results from a forthcoming work on the analysis of Markov chains using the representation theory of $S_n$. First, we give explicit formulas for the decompositions of tensor powers of the defining and standard representations of $S_n$.  Secondly, we prove that any Markov chain on $S_n$ starting with one fixed point and whose increment distributions are class measures will always average exactly one fixed point.\end{abstract}
\maketitle

The defining, or permutation, representation of $S_n$ is the $n$-dimensional representation $\varrho$ where
\begin{equation}
  \ (\varrho(\sigma))_{i, j} = \begin{cases}
		1  &\text{$\sigma(j)=i$}\\
		0 			 &\text{otherwise}.
    \end{cases}
\end{equation}
Since the fixed points of $\sigma$ can be read off of the matrix diagonal, the character of $\varrho$ at $\sigma$, $\chi_{\varrho}(\sigma)$, is precisely the number of fixed points of $\sigma$.  The irreducible representations of $S_n$ are parametrized by the partitions of $n$, and $\varrho$ decomposes as $S^{(n-1, 1)} \oplus S^{(n)}$.  Note that $\chi_{S^{(n-1, 1)}}(\sigma)$ is one less than the number of fixed points of $\sigma$.  In the terminology of \cite{FH91}, we call the $(n-1)$-dimensional irrep $S^{(n-1, 1)}$ the standard representation of $S_n$.

Our first proposition gives a nice formula for the decomposition of tensor powers of $\varrho$ into irreps, i.e. the coefficients $a_{\lambda, r}$ in the expression 
\begin{equation}
\varrho^{\otimes r} = \underset{\lambda \vdash n}{\bigoplus} a_{\lambda, r} S^{\lambda} := \underset{\lambda \vdash n}{\bigoplus} (S^{\lambda})^{\oplus a_{\lambda, r}}.
\end{equation}

\begin{proposition}\label{PrIrreps}
Let $\lambda \vdash n$ and $1 \le r \le n-\lambda_2$.  The multiplicity of $S^{\lambda}$ in the irreducible representation decomposition of $\varrho^{\otimes r}$ is given by
\begin{equation}\label{EqIrreps}
a_{\lambda, r} = f^{\bar{\lambda}} \sum_{i=\abs{\bar{\lambda}}}^r \binom{i}{\abs{\bar{\lambda}}}\stirling{r}{i},
\end{equation}
where $\bar{\lambda}=(\lambda_2, \lambda_3, \ldots)$ with weight $\abs{\bar{\lambda}}$, $f^{\bar{\lambda}}$ is the number of standard Young tableaux of shape $\bar{\lambda}$, and $\stirling{r}{i}$ is a Stirling number of the second kind. 
\end{proposition}

\begin{proof}
The heavy lifting had already been done by Goupil and Chauve, who derived in \cite{GC06} the generating function
\begin{equation}\label{EqTensorGF}
\sum_{r \ge \abs{\bar{\lambda}}} a_{\lambda, r} \frac{x^r}{r!} = \frac{f^{\bar{\lambda}}}{\abs{\bar{\lambda}}!}e^{e^x-1}(e^x-1)^{\abs{\bar{\lambda}}}.
\end{equation}
By (24b) and (24f) in Chapter 1 of \cite{Sta97},
\begin{equation}\label{EqStanley}
\sum_{s \ge j} \stirling{s}{j} \frac{x^s}{s!} = \frac{(e^x-1)^j}{j!}
\end{equation} 
and
\begin{equation}
\sum_{t \ge 0} B_t \frac{x^t}{t!} = e^{e^x-1},
\end{equation}
where $B_0:=1$ and $B_t = \sum_{q=1}^t \stirling{t}{q}$ is the $t$-th Bell number, so we obtain from (\ref{EqTensorGF}) that
\begin{equation}
\frac{a_{\lambda, r}}{r!}= f^{\bar{\lambda}}\sum_{s+t=r} \frac{B_t}{s!t!}\stirling{s}{\abs{\bar{\lambda}}},
\end{equation}
and thus
\begin{equation}
\begin{split}
\frac{a_{\lambda, r}}{f^{\bar{\lambda}}} &= \sum_{t=0}^{r-\abs{\bar{\lambda}}} B_t\binom{r}{t}\stirling{r-t}{\abs{\bar{\lambda}}} \\ &= \stirling{r}{\abs{\bar{\lambda}}} + \sum_{t=1}^{r-\abs{\bar{\lambda}}} \sum_{q=1}^t \stirling{t}{q}\binom{r}{t}\stirling{r-t}{\abs{\bar{\lambda}}} \\ &= \stirling{r}{\abs{\bar{\lambda}}} + \sum_{q=1}^{r-\abs{\bar{\lambda}}} \sum_{t=q}^{r-\abs{\bar{\lambda}}} \stirling{t}{q}\binom{r}{t}\stirling{r-t}{\abs{\bar{\lambda}}}.
\end{split}
\end{equation}
By (24.1.3, II.A) of \cite{AS65},
\begin{equation}
\sum_{t=q}^{r-\abs{\bar{\lambda}}} \stirling{t}{q}\binom{r}{t}\stirling{r-t}{\abs{\bar{\lambda}}} = \binom{q+\abs{\bar{\lambda}}}{\abs{\bar{\lambda}}}\stirling{r}{q+\abs{\bar{\lambda}}},
\end{equation} 
so that
\begin{equation}
\begin{split}
\frac{a_{\lambda, r}}{f^{\bar{\lambda}}} &= \stirling{r}{\abs{\bar{\lambda}}} + \sum_{q=1}^{r-\abs{\bar{\lambda}}} \binom{q+\abs{\bar{\lambda}}}{\abs{\bar{\lambda}}}\stirling{r}{q+\abs{\bar{\lambda}}} \\ &= \stirling{r}{\abs{\bar{\lambda}}} + \sum_{i=\abs{\bar{\lambda}}+1}^{r} \binom{i}{\abs{\bar{\lambda}}}\stirling{r}{i} = \sum_{i=\abs{\bar{\lambda}}}^{r} \binom{i}{\abs{\bar{\lambda}}}\stirling{r}{i},
\end{split}
\end{equation}
as was to be shown.
\end{proof}

Now, let $b_{\lambda, r}$ be the multiplicities such that
\begin{equation}
(S^{(n-1, 1)})^{\otimes r} = \underset{\lambda \vdash n}{\bigoplus} b_{\lambda, r} S^{\lambda}.
\end{equation}
Goupil and Chauve also derived the generating function 
\begin{equation}\label{EqTensorGF2}
\sum_{r \ge \abs{\bar{\lambda}}} b_{\lambda, r} \frac{x^r}{r!} = \frac{f^{\bar{\lambda}}}{\abs{\bar{\lambda}}!}e^{e^x-x-1}(e^x-1)^{\abs{\bar{\lambda}}}, 
\end{equation}
so from Proposition \ref{PrIrreps} we can obtain a formula for the decomposition of $(S^{(n-1, 1)})^{\otimes r}$ as well.
\begin{corollary}\label{CoIrreps}
Let $\lambda \vdash n$ and $1 \le r \le n-\lambda_2$.  The multiplicity of $S^{\lambda}$ in the irreducible representation decomposition of $(S^{(n-1, 1)})^{\otimes r}$ is given by
\begin{equation}
b_{\lambda, r} = f^{\bar{\lambda}}\sum_{s=\abs{\bar{\lambda}}}^r (-1)^{r-s} \binom{r}{s}\left( \sum_{i=\abs{\bar{\lambda}}}^s \binom{i}{\abs{\bar{\lambda}}}\stirling{s}{i}\right).
\end{equation}
\end{corollary}  
\begin{proof}
Comparing (\ref{EqTensorGF2}) with (\ref{EqTensorGF}) gives
\begin{equation}
\sum_{r \ge \abs{\bar{\lambda}}} b_{\lambda, r} \frac{x^r}{r!} = \left(\sum_{s \ge \abs{\bar{\lambda}}} a_{\lambda, s} \frac{x^s}{s!}\right) e^{-x} = \left(\sum_{s \ge \abs{\bar{\lambda}}} a_{\lambda, s} \frac{x^s}{s!}\right) \left(\sum_{t \ge 0}\frac{(-x)^t}{t!}\right),
\end{equation}
so that 
\begin{equation}
\frac{b_{\lambda, r}}{r!}= \sum_{s+t=r} \frac{(-1)^t a_{\lambda, s}}{s!t!} = \sum_{s=\abs{\bar{\lambda}}}^r \frac{(-1)^{r-s} }{s!(r-s)!}\left(f^{\bar{\lambda}} \sum_{i=\abs{\bar{\lambda}}}^s \binom{i}{\abs{\bar{\lambda}}}\stirling{s}{i}\right),
\end{equation}
and the result follows.
\end{proof}
\begin{remark}
Corollary \ref{CoIrreps} is very similar to Proposition 2 of \cite{GC06}, but our result is slightly cleaner, as it does not involve associated Stirling numbers of the second kind.  
\end{remark}

For our second proposition, we use $S^{(n-1, 1)}$ to prove a martingale-like property about the number of fixed points for certain Markov chains on $S_n$.  Before doing so, however, some preliminaries are in order.

Let $\mu$ be a measure on a finite group $G$.  The Fourier transform of $\mu$ is a matrix-valued map on the irreps of $G$ defined by $\hat{\mu}(\rho) = \sum_{g\in G} \mu(g)\rho(g)$.  The convolution of two measures $\mu$ and $\nu$ on $G$ is the measure defined by
\begin{equation}
\mu \ast \nu = \sum_{h\in G}\mu(gh^{-1})\nu(h), 
\end{equation}
and the Fourier transform transforms convolutions to pointwise products: $\widehat{\mu \ast \nu}=\hat{\mu}\hat{\nu}$.  If $\mu$ is a class measure, then for every irrep $\rho$ of $G$, we have that
\begin{equation}\label{EqClassMeas}
\hat{\mu}(\rho)=\left(\frac{1}{d_{\rho}}\sum_g \mu(g)\chi_{\rho}(g)\right)I_{d_\rho},
\end{equation}
where $d_{\rho}$ is the dimension of $\rho$.  For a detailed introduction to non-commutative Fourier analysis in the context of Markov chain theory, see Chapter 16 of \cite{Beh00}.

Let $E_{\mu}$ denote expectation with respect to $\mu$, and let $\rho$ be an irrep of $S_n$, then as observed in Chapter 3D of \cite{Dia88}, 
\begin{equation}\label{EqDia}
E_{\mu}(\chi_{\rho})=\sum_{\sigma\in S_n}\mu(\sigma)\text{tr}(\rho(\sigma)) = \text{tr}\left(\sum_{\sigma\in S_n}\mu(\sigma)\rho(\sigma)\right) = \text{tr}(\hat{\mu}(\rho)).
\end{equation}
We can now state and prove the proposition, which says that if a Markov chain on $S_n$ whose increment distributions are class measures starts with one fixed point, then it will always average exactly one fixed point.

\begin{proposition}\label{PrOneFP}
Form  Markov chain $\{X_i\}$ on $S_n$ as follows: let $X_0$ be the identity, and set $X_1=\tau_1 X_0$, where $\tau_1$ is selected according to any class measure supported on the set of permutations with one fixed point.  For $k \ge 2$, set $X_k = \tau_k X_{k-1}$, where $\tau_k$ is selected according to any class measure on $S_n$ (the measure can be different for each $k$).  Then the expected number of fixed points of $X_k$ is one for all $k \ge 1$.
\end{proposition}
\begin{proof}
Let $\nu_1$ be a class measure supported on the set of permutations with one fixed point, $\nu_2, \nu_3, \ldots, \nu_k$ be class measures on $S_n$, and define $\mu_k = \nu_k \ast \cdots \ast \nu_2 \ast \nu_1$.  By (\ref{EqDia}), 
\begin{equation}
E_{\mu_k} (\chi_{S^{(n-1, 1)}})=\text{tr}[\widehat{\mu_k}(S^{(n-1,1)})]=\text{tr}[\widehat{\nu_1}(S^{(n-1, 1)})\widehat{\nu_2}(S^{(n-1, 1)}) \cdots \widehat{\nu_k}(S^{(n-1, 1)})],
\end{equation}
where 
\begin{equation}
\widehat{\nu_1}(S^{(n-1, 1)})= \left(\frac{1}{n-1}\sum_{\sigma \in S_n} \nu_1(\sigma)\chi_{S^{(n-1, 1)}}(\sigma)\right)I_{n-1}
\end{equation}
by (\ref{EqClassMeas}).
Consider the anatomy of the partition $(n-1, 1)$: under the Murnaghan-Nakayama rule (see Theorem 4.10.2 of \cite{Sag10}), the only way for a single box to remain at the end is for the box in the second row to have been removed as a singleton, which requires a cycle type with at least two fixed points.  This means that $\chi_{S^{(n-1, 1)}}(\sigma)=0$ if $\sigma$ has one fixed point.  On the other hand, if $\sigma$ does not have exactly one fixed point, then $\nu_1(\sigma)=0$.  Thus $\widehat{\nu_1}(S^{(n-1, 1)})=\mathbf{0}$, which in turn implies that $E_{\mu_k} (\chi_{S^{(n-1, 1)}})=0$, and hence the expected number of fixed points with respect to $\mu_k$ is one for all $k \ge 1$.
\end{proof}

\end{document}